\documentclass{amsart}

\usepackage{amssymb,amsmath,amsthm,latexsym,booktabs,tikz}

\theoremstyle{definition}
\newtheorem{definition}{Definition}[section]

\newtheorem{remark}[definition]{Remark}
\newtheorem{algorithm}[definition]{Algorithm}

\theoremstyle{plain}
\newtheorem{lemma}[definition]{Lemma}
\newtheorem{proposition}[definition]{Proposition}
\newtheorem{theorem}[definition]{Theorem}
\newtheorem{corollary}[definition]{Corollary}
\newtheorem{conjecture}[definition]{Conjecture}

\newcommand{\TT}{\mathrm{TT}}
\newcommand{\DD}{\mathrm{DD}}

\allowdisplaybreaks

\begin{document}

\title[Dendriform analogues of Lie and Jordan triple systems]
{Dendriform analogues of\\Lie and Jordan triple systems}

\author{Murray R. Bremner}

\address{Department of Mathematics and Statistics, University of Saskatchewan, Canada}

\email{bremner@math.usask.ca}

\author{Sara Madariaga}

\address{Department of Mathematics and Statistics, University of Saskatchewan, Canada}

\email{madariaga@math.usask.ca}




\begin{abstract}
We use computer algebra to determine all the multilinear polynomial identities of degree $\le 7$
satisfied by the trilinear operations $(a \cdot b) \cdot c$ and $a \cdot (b \cdot c)$ in the
free dendriform dialgebra, where $a \cdot b$ is the pre-Lie or the pre-Jordan product.
For the pre-Lie triple products, we obtain one identity in degree 3, and three independent identities
in degree 5, and we show that every identity in degree 7 follows from the identities of lower degree.
For the pre-Jordan triple products, there are no identities in degree 3, five independent identities
in degree 5, and ten independent irreducible identities in degree 7.
Our methods involve linear algebra on large matrices over finite fields,
and the representation theory of the symmetric group.
\end{abstract}

\maketitle

\allowdisplaybreaks


\section{Introduction}

This paper is motivated by the desire to discover natural analogues of Lie and Jordan triple systems in 
the setting of dendriform dialgebras.
Jacobson \cite{Jacobson1949} introduced Lie and Jordan triple systems to study subspaces of associative algebras 
closed under either the triple commutator $[[x,y],z]$ (the Lie triple product) where $[a,b] = ab - ba$,
or the triple anticommutator $( a \circ b ) \circ c$ where $a \circ b = ab + ba$
(the Jordan triple product is now defined to be $a \circ (b \circ c) - b \circ (a \circ c) + (a \circ b) \circ c$).

Loday \cite{Loday1995} introduced associative dialgebras and their Koszul duals, the dendriform dialgebras, 
which are closely related to the Hopf algebra of binary trees;
see Loday and Ronco \cite{LodayRonco1998}, Loday \cite{Loday2001}.
Dendriform dialgebras provide a natural setting for pre-Lie algebras, which have been studied since the early 1960s; 
pre-Jordan algebras were introduced recently by Hou, Ni and Bai \cite{HouNiBai2013}.

Kolesnikov \cite{Kolesnikov2008} discovered a procedure for defining (binary) nonassociative analogues
of associative dialgebras; Pozhidaev \cite{Pozhidaev2010} extended this to the $n$-ary case.
Their work was reformulated by Bremner, Felipe and S\'anchez-Ortega \cite{BFSO2012} in the KP algorithm:
the input is a set of multilinear identities defining a variety of multi-operator algebras;
the output is the corresponding identities for dialgebras.
Gubarev and Kolesnikov \cite{GubarevKolesnikov2013} defined analogues of (binary) nonassociative
algebras in the setting of dendriform dialgebras.

In the present paper we use computer algebra to determine the multilinear polynomial identities of degree $\le 7$
for the trilinear operations $(a \cdot b) \cdot c$ and $a \cdot (b \cdot c)$ in the
free dendriform dialgebra, where $a \cdot b$ is the pre-Lie or pre-Jordan product.
In \S\ref{theoreticalsection} we recall definitions and basic results for dendriform dialgebras
and nonassociative multi-operator algebras.
In \S\ref{dualitysection} we determine the relation between the analogues of 
the Lie and Jordan triple products in associative and dendriform dialgebras.
In~\S\ref{preliesection}, we determine the polynomial identities of degree $\le 7$ relating 
the trilinear operations $(a \cdot b) \cdot c$ and $a \cdot (b \cdot c)$ in the free dendriform algebra, 
where $a \cdot b = a \prec b - b \succ a$ is the pre-Lie product;
we obtain one identity in degree 3, and three independent identities
in degree 5, and we show that every identity in degree 7 follows from the identities of lower degree.
In \S\ref{prejordansection}, we obtain analogous results for the pre-Jordan triple product
$a \cdot b = a \succ b + b \prec a$;
there are no identities in degree 3, five independent identities
in degree 5, and ten independent irreducible identities in degree~7.

Our methods rely on linear algebra for large matrices over finite fields,
and the representation theory of the symmetric group.
Our calculations were performed using the computer algebra system \texttt{Maple}. 
Unless indicated, the base field $\mathbb{F}$ has characteristic 0, so that
any polynomial identity is equivalent to a set of multilinear identities; 
see Zhevlakov et al.~\cite[Ch.~1]{Zhevlakov1982}.
Multilinear identities can be regarded as elements of the direct sum of copies of 
the group algebra $\mathbb{F} S_n$ of the symmetric group,
which is semisimple over a field of characteristic 0 or $p > n$.
We can therefore save computer memory by using modular arithmetic, as long as we use a prime
greater than the degree of the identities.


\section{Free nonassociative structures} \label{theoreticalsection}

We study the $S_n$-module structure of the kernel of a linear map, called the expansion map, 
from the multilinear subspace of degree $n$ in the free algebra with two ternary operations
to the corresponding subspace in the free dendriform dialgebra.
We begin with some general definitions related to free multioperator algebras.

\begin{definition}
Let $\Omega$ be a set of operation symbols of various arities,
let $X$ be a set of variables, 
and let $\mathbb{F}(\Omega,X)$ be the corresponding free algebra over $\mathbb{F}$.
Let $A$ be an algebra in the variety of all $\Omega$-algebras, and let
$f(x_1, \dots, x_n)$ be in the multilinear subspace $\mathbb{F}(\Omega,X)_n$ of degree $n$.
We say that $f$ is a \textbf{polynomial identity} for $A$ if $f(a_1, \dots, a_n) = 0$ for all $a_1, \dots, a_n \in A$, 
and we write $f(x_1, \dots, x_n) \equiv 0$.
The space $\mathbb{F}(\Omega,X)_n$ is a module over the symmetric group $S_n$ acting by permutations of the variables:
$\sigma \cdot x_i = x_{\sigma(i)}$, $1 \le i \le n$.
The subspace of identities satisfied by an algebra $A$ is an $S_n$-submodule.
\end{definition}

\begin{definition}
Given $f, f_1, \dots, f_k \in F(\Omega,X)_n$, 
we say that $f$ is a \textbf{consequence} of $f_1, \dots, f_k$
if $f$ belongs to the $S_n$-module generated by $f_1, \dots, f_k$.
In other words, $f$ is a linear combination of permutations of $f_1, \dots, f_k$.
\end{definition}

\begin{definition}
Let $\TT(X) = \mathbb{F}(\Omega,X)$ be the \textbf{free TT-algebra} with two ternary operations
$\Omega = \{ \omega_1, \omega_2 \}$; we write $[x,y,z]_i = \omega_i(x,y,z)$, $i = 1, 2$.
A basis of $\TT(X)$ consists of the \textbf{TT-monomials} obtained by applying the operations 
$[ \ast, \ast, \ast ]_i$ to $X$.
A \textbf{TT-type} is a placement of operations in a TT-monomial, ignoring the variables.
We write $\TT_n = \mathbb{F}(\Omega,X)_n$ for the multilinear subspace of degree $n$.
\end{definition}

\begin{algorithm} \label{TTtypesalgorithm}
The total order of the TT-types in (odd) degree $n$ is determined by the following recursive algorithm:
  \begin{itemize}
  \item
  Set $\texttt{TT}[ n ] \leftarrow [ \, ]$ (empty list).
  \item
  For $i = 1, 3, \dots, n{-}2$ do for $j = 1, 3, \dots, n{-}i{-}1$ do:
  \item[] \quad
  For $x \in \texttt{TT}[ n{-}i{-}j ]$ do 
  for $y \in \texttt{TT}[ j ]$ do 
  for $z \in \texttt{TT}[ i ]$ do:
  \item[] \quad \quad
  Append $[x,y,z]_1$ and $[x,y,z]_2$ to $\texttt{TT}[n]$. 
  \end{itemize}
The multilinear TT-monomials form a basis of $\TT_n$ and are ordered first by TT-type and then by lex order 
of the permutation of the variables. 
\end{algorithm}

\begin{definition}
From $f \in \TT_n$ we obtain $2(n{+}3)$ \textbf{liftings} of $f$ in $\TT_{n+2}$, 
which generate the $S_{n+2}$-submodule of consequences of $f$ in degree $n{+}2$.
Using $x_{n+1}, x_{n+2}$ we perform $n$ substitutions and 3 multiplications for $i = 1,2$:
  \begin{align*}
  &
  f( \, [ \, x_1, \, x_{n+1}, \, x_{n+2} \, ]_i, \, x_2, \, \dots, \, x_n \, ),
  \; \dots, \;
  f( \, x_1, \, \dots, \, [ \, x_k, \, x_{n+1}, \, x_{n+2} \, ]_i, \, \dots, \, x_n \, ),
  \\
  &
  \qquad \dots, \;
  f( \, x_1, \, \dots, \, x_{n-1}, \, [ \, x_n, \, x_{n+1}, \, x_{n+2} \, ]_i \, ),
  \\
  &
  [ \, f( \, x_1, \, x_2, \, \dots, \, x_n \, ), \, x_{n+1}, \, x_{n+2} \, ]_i,  
  \qquad
  [ \, x_{n+1}, \, f( \, x_1, \, x_2, \, \dots, \, x_n \, ), \, x_{n+2} \, ]_i,  
  \\
  &
  [ \, x_{n+1}, \, x_{n+2}, \, f( \, x_1, \, x_2, \, \dots, \, x_n \, ) \, ]_i.
  \end{align*}
\end{definition}

\begin{definition}
A \textbf{dialgebra} is an algebra with two binary operations $\Omega = \{ \prec, \succ \}$.
\end{definition}

\begin{definition} Leroux \cite{Leroux2003,Leroux2011}.
An \textbf{L-algebra} is an algebra with two binary operations $\Omega = \{ \prec, \succ \}$
satisfying \textbf{inner associativity}:
  \begin{equation}
  \label{L-identity}
  (x \succ y) \prec z \equiv x \succ (y \prec z).
  \end{equation}
\end{definition}

\begin{definition} Loday \cite{Loday1995}.
A \textbf{dendriform dialgebra} is an L-algebra satisfying:
  \begin{align}
  \label{dendriform1}
  (x \prec y) \prec z &\equiv x \prec (y \prec z) + x \prec (y \succ z),
  \\
  \label{dendriform2}
  x \succ (y \succ z) &\equiv (x \succ y) \succ z + (x \prec y) \succ z.
  \end{align}
\end{definition}

\begin{definition} Bokut, Chen, Huang \cite{BCH2013}.
Let $u$ be a monomial in the operations $\prec$, $\succ$ and the variables $X$.
We say that $u$ is a \textbf{normal L-monomial} if:
  \begin{itemize}
  \item
  $u \in X$, or
  \item
  $u = v \succ w$ where $v$ and $w$ are normal L-monomials, or
  \item
  $u = v \prec w$ where $v = v_1 \prec v_2$ and $v_1$, $v_2$, $v$, $w$ are normal L-monomials.
  \end{itemize}
We write $N(X)$ for the set of all normal L-monomials.
\end{definition}

\begin{theorem} \emph{Chen, Wang \cite{ChenWang2010}.} \label{normaltheorem}
The set $S$ is a Gr\"obner-Shirshov basis for the $T$-ideal generated by 
identities \eqref{dendriform1}-\eqref{dendriform2} in the free L-algebra, where
  \begin{align*}
  &
  S = \{ \, f_1(x,y,z), \, f_2(x,y,z), \, f_3(x,y,z,v) \mid x,y,z,v \in N(X) \, \},
  \\
  &
  f_1(x,y,z) = (x \prec y) \prec z - x \prec (y \prec z) - x \prec (y \succ z),
  \\
  &
  f_2(x,y,z) = (x \prec y) \succ z + (x \succ y) \succ z - x \succ (y \succ z),
  \\
  &
  f_3(x,y,z,v) = ((x \succ y) \succ z) \succ v - (x \succ y) \succ (z \succ v) + (x \succ (y \prec z)) \succ v.
  \end{align*}
\end{theorem}

\begin{remark}
Madariaga \cite{Madariaga2013} has obtained a simpler basis using the Composition-Diamond Lemma
for non-symmetric operads from Dotsenko and Vallette \cite{DotsenkoVallette2013}.
\end{remark}

\begin{definition} Chen, Wang \cite{ChenWang2010}.
A monomial $u$ in the operations $\prec$, $\succ$ and the variables $X$ 
is a \textbf{normal DD-monomial} if:
  \begin{itemize}
	\item $ u \in X $, or
	\item $ u = x \prec v $, where $x \in X$ and $v$ is a normal DD-monomial, or
	\item $ u = x \succ v $, where $x \in X$ and $v$ is a normal DD-monomial, or
	\item $ u = ( x \succ u_1 ) \succ u_2 $, where $x \in X$ and $u_1,u_2$ are normal DD-monomials.
  \end{itemize}
A \textbf{normal DD-type} is the placement of operations in a normal DD-monomial, 
ignoring the variables.
We write $\DD(X)$ for the free dendriform dialgebra and $\DD_n$ for the multilinear subspace of degree $n$.
\end{definition}

\begin{algorithm}
The total order of the normal DD-types in degree $n$ is determined by the following recursive algorithm,
where $\ast$ is a placeholder for a variable:
  \begin{itemize}
  \item
  Set $\texttt{DD}[ n ] \leftarrow [ \, ]$ (empty list).
  \item
  For $w \in \texttt{DD}[ n{-}1 ]$ do:
  \item[] \quad
  Append $\ast \prec w$ and $\ast \succ w$ to $\texttt{DD}[ n ]$.
  \item
  For $i = 1, \dots, k{-}2$ do for $v \in \texttt{DD}[ i ]$ do for $w \in \texttt{DD}[ k{-}1{-}i ]$ do:
  \item[] \quad
  Append $( \ast \succ v ) \succ w$ to $\texttt{DD}[ n ]$.
  \end{itemize}
The multilinear DD-monomials form a basis for $\DD_n$ and are ordered first 
by normal DD-type and then by lex order of the permutation of the variables. 
\end{algorithm}

\begin{remark}
Chen and Wang \cite[Cor.~3.5]{ChenWang2010} proved that the normal DD-monomials form a basis of 
the free dendriform dialgebra.
Bremner and Madariaga~\cite{BremnerMadariaga2013} gave an algorithm to rewrite any dialgebra monomial 
in terms of normal DD-monomials.
\end{remark}

\begin{definition}
Let $\mu_i(x,y,z) \in \DD_3$ ($i = 1,2$) be trilinear operations in the free dendriform dialgebra.
For all (odd) $n$, the \textbf{expansion map} $\mathcal{E}_n \colon \TT_n \to \DD_n$
is defined by replacing each occurrence of $[x,y,z]_i$ by $\mu_i(x,y,z)$, $i = 1,2$.
The \textbf{expansion matrix} $E_n$ represents $\mathcal{E}_n$ using the ordered bases defined above.
The (nonzero) elements of the kernel of $\mathcal{E}_n$ are the polynomial identities of degree $n$ 
satisfied by the operations $\mu_1$, $\mu_2$ in the free dendriform dialgebra.
\end{definition}


\section{Analogues of Lie and Jordan triple products} \label{dualitysection}

In an associative algebra, the Lie and Jordan triple products are defined by
  \begin{align*}
  [a,b,c] &= [[a,b],c] = abc - bac - cab + cba,
  \\
  \{a,b,c\} &= a \circ (b \circ c) - b \circ (a \circ c) + (a \circ b) \circ c = 2( abc + cba ),
  \end{align*}
and these two trilinear operations are related by the equation
  \begin{equation}
  \label{LJTSequation}
  [a,b,c] = \frac12 \big( \{a,b,c\} - \{b,a,c\} \big),
  \end{equation}
which accounts for the classical connection between Lie and Jordan triple systems.

The polynomial identities defining associative and dendriform dialgebras involve only the identity permutation 
of the variables.
However, we will study two trilinear operations which generalize the Lie and Jordan triple products, and hence 
involve permutations.
We consider the set $\Omega = \{ \circ_1, \circ_2 \}$ of two binary operation symbols, and we write $\mathsf{BB}_n$
for the multilinear subspace of degree $n$ in the free algebra $\mathbb{F}(\Omega,X)$ where $X = \{x_1,\dots,x_n\}$.
The $S_3$-module $\mathsf{BB}_3$ has dimension 48 and the following basis: 
  \[
  \big\{ 
  \;
  ( x_{\sigma(1)} \circ_i x_{\sigma(2)} ) \circ_j x_{\sigma(3)}, 
  \;
  x_{\sigma(1)} \circ_i ( x_{\sigma(2)} \circ_j x_{\sigma(3)} )
  \mid
  i, j \in \{1,2\}, 
  \;
  \sigma \in S_3
  \; 
  \big\}.
  \]
The following lemmas can be proved by direct calculation.

\begin{lemma}
The $S_3$-submodule $\mathsf{Dias}_3 \subset \mathsf{BB}_3$, 
spanned by the permutations of the relations defining associative dialgebras, has dimension 30.

The $S_3$-submodule $\mathsf{Dend}_3 \subset \mathsf{BB}_3$, 
spanned by the permutations of the relations defining dendriform dialgebras, has dimension 18.
\end{lemma}

\begin{definition}
In an $\Omega$-algebra, 
the \textbf{pre-Lie product} $a \bullet_l b$ and the \textbf{pre-Jordan product} $a \bullet_j b$
are the following binary operations:
  \[
  a \bullet_l b = a \circ_1 b - b \circ_2 a,
  \qquad
  a \bullet_j b = a \circ_2 b + b \circ_1 a.  
  \]
The \textbf{pre-Lie triple products} are
  \[
  ( a \bullet_l b ) \bullet_l c, \qquad a \bullet_l ( b \bullet_l c);
  \]
similarly, the \textbf{pre-Jordan triple products} are
  \[
  ( a \bullet_j b ) \bullet_j c, \qquad a \bullet_j ( b \bullet_j c).
  \]
\end{definition}

\begin{lemma} \label{badlemma}
Let $\mathrm{TLie}_3$ and $\mathrm{TJor}_3$ be the $S_3$-submodules of $\mathsf{BB}_3$ generated by 
the pre-Lie and pre-Jordan triple products, respectively.
We have 
  \begin{alignat*}{3}
  \dim \mathrm{TLie}_3 &= 12,
  &\quad
  \dim( \, \mathsf{Dias}_3 + \mathrm{TLie}_3 \, ) &= 39, 
  &\quad
  \dim( \, \mathsf{Dias}_3 \cap \mathrm{TLie}_3 \, ) &= 3,
  \\
  &&
  \dim( \, \mathsf{Dend}_3 + \mathrm{TLie}_3 \, ) &= 30, 
  &\quad
  \dim( \, \mathsf{Dend}_3 \cap \mathrm{TLie}_3 \, ) &= 0,
  \\
  \dim \mathrm{TJor}_3 &= 12,
  &\quad
  \dim( \, \mathsf{Dias}_3 + \mathrm{TJor}_3 \, ) &= 42, 
  &\quad
  \dim( \, \mathsf{Dias}_3 \cap \mathrm{TJor}_3 \, ) &= 0,
  \\
  &&
  \dim( \, \mathsf{Dend}_3 + \mathrm{TJor}_3 \, ) &= 30, 
  &\quad
  \dim( \, \mathsf{Dend}_3 \cap \mathrm{TJor}_3 \, ) &= 0,
  \end{alignat*}
Furthermore,
  \[
  \mathsf{Dias}_3 + \mathrm{TLie}_3 + \mathrm{TJor}_3 
  = 
  \mathsf{Dias}_3 + \mathrm{TJor}_3,
  \]  
so that in the free diassociative dialgebra, the $S_3$-submodule generated by the pre-Lie triple products
is a submodule of that generated by the pre-Jordan triple products:
  \[
  ( \mathsf{Dias}_3 + \mathrm{TLie}_3 ) / \mathsf{Dias}_3
  \subset
  ( \mathsf{Dias}_3 + \mathrm{TJor}_3 ) / \mathsf{Dias}_3.
  \]
Hence in this case, the pre-Lie triple products can be expressed in terms of the pre-Jordan triple products,
generalizing equation \eqref{LJTSequation}.
On the other hand,
  \[
  \dim( \, \mathsf{Dend}_3 + \mathrm{TLie}_3 + \mathrm{TJor}_3 \, ) = 36,
  \]  
so that in the free dendriform dialgebra, the two submodules $\mathrm{TLie}$ and $\mathrm{TJor}$ intersect,
but there is no containment between them,
  \[
  ( \mathsf{Dend}_3 + \mathrm{TLie}_3 ) / \mathsf{Dias}_3
  \not \subset
  ( \mathsf{Dend}_3 + \mathrm{TJor}_3 ) / \mathsf{Dias}_3,
  \]
analogous to equation \eqref{LJTSequation}.
\end{lemma}

The analogues of Lie and Jordan triple systems in the setting of associative dialgebras have been studied
by Bremner and S\'anchez-Ortega \cite{BSO2011} and Bremner, Felipe and S\'anchez-Ortega \cite{BFSO2012}.
Those papers demonstrated a connection between the resulting structures, Leibniz triple systems and 
Jordan triple disystems, analogous to the connection in the classical case based on equation \eqref{LJTSequation}.
However, from the last statement of Lemma \ref{badlemma}, we see that no such connection is to be expected
in the setting of dendriform dialgebras.


\section{Pre-Lie triple systems} \label{preliesection}

In the rest of this paper, we specialize to the case of the free dendriform dialgebra, and 
hence we can simplify our notation.
The pre-Lie product will be denoted $a \cdot b = a \prec b - b \succ a$,
and the pre-Lie triple products will be denoted
  \begin{align*}
  [ a, b, c ]_1
  &=
  ( a \cdot b ) \cdot c
  =
  ( a \prec b ) \prec c - ( b \succ a ) \prec c - c \succ ( a \prec b ) + c \succ ( b \succ a ),
  \\
  [ a, b, c ]_2
  &=
  a \cdot ( b \cdot c )
  =
  a \prec ( b \prec c ) - a \prec ( c \succ b ) - ( b \prec c ) \succ a + ( c \succ b ) \succ a.
  \end{align*}

\begin{definition}
The \textbf{right-symmetric identity} says that the associator is invariant under transposition 
of the second and third arguments:
  \[
  ( a \cdot b ) \cdot c - a \cdot ( b \cdot c ) \equiv ( a \cdot c ) \cdot b - a \cdot ( c \cdot b ).
  \]
Using trilinear operations, this identity is the \textbf{ternary right-symmetric identity}:
  \[
  [a,b,c]_1 - [a,b,c]_2 \equiv [a,c,b]_1 - [a,c,b]_2.
  \]
\end{definition}

\begin{lemma} \label{PL3}
The pre-Lie product in every dendriform dialgebra satisfies the right-symmetric identity.
Every identity in degree 3 satisfied by the pre-Lie triple products
in every dendriform dialgebra follows from the ternary right-symmetric identity.
\end{lemma}

\begin{proof}
The first statement can be verified by direct calculation.
For the second, we note that applying 6 permutations of 3 variables to $[a,b,c]_i$ ($i=1,2$) 
gives 12 monomials forming an ordered basis of $\TT_3$.
There are 8 association types in the free nonassociative dialgebra, 
$(a \circ_1 b) \circ_2 c$ and $a \circ_1 (b \circ_2 c)$
where $\circ_1, \circ_2 \in \{ \prec, \succ \}$.
The dendriform identities \eqref{L-identity}-\eqref{dendriform2} eliminate 3 types, leaving 5 normal DD-types:
  \[
  a \prec ( b \prec c ), \quad   
  a \succ ( b \prec c ), \quad
  a \prec ( b \succ c ), \quad
  a \succ ( b \succ c ), \quad
  ( a \succ b ) \succ c.
  \]
Applying 6 permutations gives 30 monomials forming an ordered basis of $\DD_3$.
The identities relating the pre-Lie triple products are the nonzero elements of 
the kernel of the expansion map $\mathcal{E}_3\colon \TT_3 \to \DD_3$.
In the expansion matrix $E_3$, the $(i,j)$ entry is the coefficient of the $i$-th DD-monomial in the
normalized expansion of the $j$-th TT-monomial.
Each expansion has 4 terms, but normalization gives 5 terms:
  \begin{align*}
  [ a, b, c ]_1
  &=
  a \prec (b \prec c) + a \prec (b \succ c)  - b \succ (a \prec c) - c \succ (a \prec b) + c \succ (b \succ a),
  \\
  [ a, b, c ]_2
  &=
  a \prec (b \prec c) - a \prec (c \succ b) + (b \succ c) \succ a - b \succ (c \succ a) + (c \succ b) \succ a.
  \end{align*}
A basis for the nullspace of $E_3$ consists of the coefficient vectors of the identities,
  \begin{align*}
  &
  [a, b, c]_1 - [a, c, b]_1 - [a, b, c]_2 + [a, c, b]_2 \equiv 0, \\
  &
  [b, a, c]_1 - [b, c, a]_1 - [b, a, c]_2 + [b, c, a]_2 \equiv 0, \\
  &
  [c, a, b]_1 - [c, b, a]_1 - [c, a, b]_2 + [c, b, a]_2 \equiv 0,
  \end{align*}
which are the permutations of the ternary right-symmetric identity.
\end{proof}

\begin{lemma} \label{PLlifted5lemma}
The subspace $O_5 \subset \TT_5$ consisting of identities which follow from the ternary right-symmetric identity 
has dimension 630.
\end{lemma}

\begin{proof}
There are 12 TT-types in degree 5:
  \[
  [ [ a, b, c ]_i, d, e ]_j,
  \quad
  [ a, [ b, c, d ]_i, e ]_j,
  \quad
  [ a, b, [ c, d, e ]_i ]_j,
  \quad
  i,j \in \{1,2\}.
  \]
Applying 120 permutations of 5 variables gives $\dim \TT_5 = 1440$.
We write the ternary right-symmetric identity as $\mathrm{TRS}(a,b,c) \equiv 0$ where
  \[
  \mathrm{TRS}(a,b,c) = \big( \, [a,b,c]_1 - [a,c,b]_1 \big) - \big( \, [a,b,c]_2 - [a,c,b]_2 \big).
  \]
There are 12 liftings of $\mathrm{TRS}(a,b,c)$ to degree 5:
  \begin{align*} 
  &
  \mathrm{TRS}( \, [ \, a, \, d, \, e \, ]_i, \, b, \, c \, ), \qquad 
  \mathrm{TRS}( \, a, \, [ \, b, \, d, \, e \, ]_i, \, c \, ), \qquad 
  \mathrm{TRS}( \, a, \, b, \, [ \, c, \, d, \, e \, ]_i \, ),
  \\
  & 
  [ \, \mathrm{TRS}( \, a, \, b, \, c \, ), \, d, \, e \, ]_i, \qquad 
  [ \, d, \, \mathrm{TRS}( \, a, \, b, \, c \, ), \, e \, ]_i, \qquad 
  [ \, d, \, e, \, \mathrm{TRS}( \, a, \, b, \, c \, ) \, ]_i.
  \end{align*}
Applying 120 permutations gives 1440 elements of $\TT_3$ 
which span the subspace of identities implied by $\mathrm{TRS}(a,b,c) \equiv 0$.
This subspace is the row space of a $1440 \times 1440$ matrix, 
which has rank 630 using modular arithmetic with $p = 101$.
\end{proof}

\begin{proposition} \label{PLTS5proposition}
The kernel $K_5 \subset \TT_5$ of the expansion map $\mathcal{E}_5$ has dimension 815, 
and hence the quotient module $K_5 / O_5$ has dimension 185.
\end{proposition}

\begin{proof}
In degree 5 there are $14 \cdot 2^4 = 224$ association types for two binary operations, 
which can be reduced to a set of 42 normal DD-types.
Applying 120 permutations of 5 variables gives 5040 monomials forming a basis of $\DD_5$.
We expand each of the 12 TT-types; each expansion has 16 terms, 
but after replacing each term by a linear combination of normal DD-monomials, 
the 12 expansions have respectively 41, 33, 41, 33, 32, 29, 44, 35, 27, 36, 35, 46 terms.
The expansion matrix $E_5$ has size $5040 \times 1440$ and rank 815 using modular arithmetic with $p = 101$. 
\end{proof}

It follows from Proposition \ref{PLTS5proposition} that there are new identities in degree 5.
A set of $S_5$-module generators for these new identities is given in the next theorem.

\begin{theorem} \label{PL5}
Every multilinear identity of degree 5 satisfied by the pre-Lie triple products
$[a,b,c]_1$ and $[a,b,c]_2$ in the free dendriform dialgebra is a consequence of the liftings
of $\mathrm{TRS}(a,b,c) \equiv 0$ and the following 3 identities:
  \begin{align*}
  &
  [ [ a, b, c ]_1, d, e ]_2 
  - [ [ a, d, e ]_2, b, c ]_1 
  + [ a, [ d, e, b ]_1, c ]_1 
  - [ a, [ b, d, e ]_2, c ]_1 
  \\
  &\quad
  + [ a, b, [ d, e, c ]_1 ]_1 
  - [ a, b, [ c, d, e ]_2 ]_1
  \equiv 0,
  \\
  & 
  [ [ a, b, c ]_1, d, e ]_1 
  - [ [ a, d, b ]_1, c, e ]_1 
  + [ [ a, d, c ]_1, b, e ]_1 
  - [ [ a, c, b ]_1, d, e ]_1 
  \\
  &\quad
  - [ a, [ b, c, d ]_1, e ]_1 
  + [ a, [ d, b, c ]_1, e ]_1 
  - [ a, [ d, c, b ]_1, e ]_1 
  + [ a, [ c, b, d ]_1, e ]_1 
  \equiv 0,
  \\
  &
  [ [ a, b, c ]_1, d, e ]_2 
  - [ [ a, d, e ]_2, b, c ]_1 
  + [ [ a, d, e ]_2, b, c ]_2 
  - [ [ a, b, c ]_2, d, e ]_2 
  \\
  &\quad
  + [ a, [ d, e, b ]_1, c ]_1 
  + [ a, [ d, e, c ]_1, b ]_1 
  - [ a, [ b, c, e ]_1, d ]_1 
  - [ a, [ d, e, c ]_1, b ]_2 
  \\
  &\quad
  + [ a, [ b, c, e ]_1, d ]_2 
  - [ a, [ b, d, e ]_2, c ]_1 
  - [ a, [ c, d, e ]_2, b ]_1 
  - [ a, [ d, b, c ]_2, e ]_2 
  \\
  &\quad
  + [ a, [ b, d, e ]_2, c ]_2 
  + [ a, [ c, d, e ]_2, b ]_2 
  + [ a, d, [ b, c, e ]_1 ]_1 
  - [ a, d, [ e, b, c ]_2 ]_2
  \equiv 0.
  \end{align*} 
\end{theorem}

\begin{proof}
We construct a $1560 \times 1440$ matrix $M$, with a $1440 \times 1440$ upper block of 
and a $120 \times 1440$ lower block, initialized to zero.
For each of the 12 liftings of $\mathrm{TRS}(a,b,c)$,
we fill the lower block of $M$ with the coefficient vectors
obtained by applying all 120 permutations,
and then compute the row canonical form (RCF).
We retain the results of this computation in the first 630 rows of $M$.

We compute the RCF of the expansion matrix $E_5$, and extract the canonical basis of its nullspace,
by successively setting the free variables equal to the standard basis vectors and solving for the leading variables.
Using symmetric representatives modulo $p = 101$, the components of the nullspace basis 
vectors are in $\{ -1, 0, 1 \}$.
Interpreting these as integers, we compute the squared Euclidean length of each nullspace basis
vector, and sort the vectors by increasing length.

For each of the 815 vectors in the sorted list, we fill the lower block of $M$ with 
the coefficient vectors obtained by applying all 120 permutations of the variables, and then compute the RCF.
Four vectors increase the rank: first to 745, then to 770, then to 800, and finally to 815.
Further computations show that vector 3 belongs to the $S_5$-module generated by vectors 1, 2, 4;
this is the only dependence relation among the four vectors. 
Vectors 1, 2, 4 are the coefficient vectors of the identities in the statement of this theorem,
which are generators of the $S_5$-module $K_5 / O_5$.
\end{proof}

\begin{remark}
In principle, the identities of Theorem \ref{PL5} could be checked directly by expanding the terms 
in the free dendriform dialgebra and computing the normal form of each term as a linear combination of
normal DD-monomials.
However, even for the first identity, with only 6 terms, the expansions produce 96 terms, 
and the normalizations have $33 + 41 + 32 + 44 + 27 + 35 = 212$ terms.
Hence this verification is only practical with the assistance of a computer algebra system.
\end{remark}

\begin{theorem} \label{theoremPL7}
Every polynomial identity of degree 7 satisfied by the pre-Lie triple products in every dendriform dialgebra
follows from the identities of lower degree.
\end{theorem}

\begin{table}
\begin{tabular}{lr|rrr|rrrr}
&\; & \multicolumn{3}{c|}{\; $L_7^\lambda$ (lifted identities)} & \multicolumn{4}{c}{\; $X_7^\lambda$ (all identities)} 
\\
$\lambda$ &\; $d_\lambda$ &\; rows &\; cols &\; rank &\; rows &\; cols &\; rank &\; null 
\\ 
\midrule
7 &\; 1 &\; 240 &\; 96 &\; 48 &\; 96 &\; 429 &\; 48 &\; 48 
\\
61 &\; 6 &\; 1440 &\; 576 &\; 362 &\; 576 &\; 2574 &\; 214 &\; 362
\\
52 &\; 14 &\; 3360 &\; 1344 &\; 923 &\; 1344 &\; 6006 &\; 421 &\; 923
\\
511 &\; 15 &\; 3600 &\; 1440 &\; 1021 &\; 1440 &\; 6435 &\; 419 &\; 1021
\\
43 &\; 14 &\; 3360 &\; 1344 &\; 962 &\; 1344 &\; 6006 &\; 382 &\; 962
\\
421 &\; 35 &\; 8400 &\; 3360 &\; 2486 &\; 3360 &\; 15015 &\; 874 &\; 2486
\\
4111 &\; 20 &\; 4800 &\; 1920 &\; 1461 &\; 1920 &\; 8580 &\; 459 &\; 1461
\\
331 &\; 21 &\; 5040 &\; 2016 &\; 1523 &\; 2016 &\; 9009 &\; 493 &\; 1523
\\
322 &\; 21 &\; 5040 &\; 2016 &\; 1542 &\; 2016 &\; 9009 &\; 474 &\; 1542
\\
3211 &\; 35 &\; 8400 &\; 3360 &\; 2615 &\; 3360 &\; 15015 &\; 745 &\; 2615
\\
31111 &\; 15 &\; 3600 &\; 1440 &\; 1146 &\; 1440 &\; 6435 &\; 294 &\; 1146
\\
2221 &\; 14 &\; 3360 &\; 1344 &\; 1063 &\; 1344 &\; 6006 &\; 281 &\; 1063
\\
22111 &\; 14 &\; 3360 &\; 1344 &\; 1083 &\; 1344 &\; 6006 &\; 261 &\; 1083
\\
211111 &\; 6 &\; 1440 &\; 576 &\; 472 &\; 576 &\; 2574 &\; 104 &\; 472 
\\
1111111 &\; 1 &\; 240 &\; 96 &\; 80 &\; 96 &\; 429 &\; 16 &\; 80
\\
\midrule
\end{tabular}
\medskip
\caption{Ranks for pre-Lie triple products in degree 7}
\label{PL7table}
\end{table}

\begin{proof}
There are 96 TT-types and 429 normal DD-types, so $\dim \TT_7 = 429 \cdot 7!$ and $\dim \DD_7 = 96 \cdot 7!$.
The $2162160 \times 483840$ expansion matrix $E_7$ is much too large to process efficiently, 
even with modular arithmetic, so we use the representation theory of the symmetric group 
to decompose the problem into smaller pieces corresponding to the irreducible representations of $S_7$;
see Bremner and Peresi \cite{BremnerPeresi2011}.
In particular, for any partition $\lambda$ of $n$, and any permutation $\pi \in S_n$, 
the representation matrix $R_\lambda(\pi)$ can be computed using the method of 
Clifton\footnote{This short and beautiful paper deserves to be better known in the representation theory and 
computer algebra communities. 
In his MathSciNet review (MR0624907), G. D. James states:
``Most methods for working out the matrices for the natural representation are messy, 
but this paper gives an approach which is simple both to prove and to apply.''}
\cite{Clifton1981}.

We first consider the identities of degree 7 implied by the identities of lower degree.
In degree 5, we have the 12 liftings of $\mathrm{TRS}(a,b,c)$ and the three new identities of Theorem \ref{PL5},
giving 15 identities generating the $S_5$-module of identities satisfied by the pre-Lie triple products. 
Each such identity $I(a,b,c,d,e)$ has 16 liftings to degree 7; 
we have 5 substitutions and 3 multiplications for $i = 1,2$:
  \begin{align*}
  &
  I( \, [ \, a, \, f, \, g \, ]_i, \, b, \, c, \, d, \, e \, ), \quad
  I( \, a, \, [ \, b, \, f, \, g \, ]_i, \, c, \, d, \, e \, ), \quad
  I( \, a, \, b, \, [ \, c, \, f, \, g \, ]_i, \, d, \, e \, ),
  \\
  &
  I( \, a, \, b, \, c, \, [ \, d, \, f, \, g \, ]_i, \, e \, ), \quad
  I( \, a, \, b, \, c, \, d, \, [ \, e, \, f, \, g \, ]_i \, ), \quad
  [ \, I( \, a, \, b, \, c, \, d, \, e \, ), \, f, \, g \, ]_i,
  \\
  &
  [ \, f, \, I( \, a, \, b, \, c, \, d, \, e \, ), \, g \, ]_i, \quad
  [ \, f, \, g, \, I( \, a, \, b, \, c, \, d, \, e \, ) \, ]_i.  
  \end{align*}
Altogether this gives 240 generators for the $S_7$-submodule $O_7 \subset \TT_7$.

Let $\lambda$ be a partition of 7, let $d_\lambda$ be the dimension of the corresponding irreducible
representation of $S_7$, and let $R_\lambda\colon \mathbb{F} S_7 \to M_{d_\lambda}(\mathbb{F})$ 
be the corresponding natural representation.  
We construct a $240 d_\lambda \times 96 d_\lambda$ matrix $L_7^\lambda$ consisting of 
$d_\lambda \times d_\lambda$ blocks.
The block in position $(i,j)$ contains the representation matrix for the terms of the $i$-th lifted identity
in the $j$-th TT-type.
We compute $\mathrm{RCF}(L_7^\lambda)$, delete any zero rows, and denote its rank by $\texttt{lifrank}(\lambda)$.
We construct a $96 d_\lambda \times 429 d_\lambda$ matrix $X_7^\lambda$ consisting of 
$d_\lambda \times d_\lambda$ blocks.
The block in position $(i,j)$ contains the representation matrix for the terms in the normalized expansion of 
the $i$-th TT-type which have the $j$-th normal DD-type.
Each expansion has 64 terms, but after normalization we obtain from 171 to 447 terms.
We compute $\mathrm{RCF}((X_7^\lambda)^t)$ and denote its rank and nullity by 
$\texttt{exprank}(\lambda)$ and $\texttt{allrank}(\lambda) = 96 d_\lambda - \texttt{exprank}(\lambda)$. 
We compute the canonical basis for its nullspace and put the basis vectors
into the rows of a $\texttt{allrank}(\lambda) \times 96 d_\lambda$ matrix $K_7^\lambda$.
The rows of $\mathrm{RCF}(K_7^\lambda)$ form the canonical basis 
for the $S_7$-module of identities satisfied by the pre-Lie triple products in partition $\lambda$.

If new identities exist then $\texttt{lifrank}(\lambda) < \texttt{allrank}(\lambda)$ 
for some $\lambda$.
On the other hand, if $\texttt{lifrank}(\lambda) = \texttt{allrank}(\lambda)$ for some $\lambda$
then we check that the matrices are equal, $\mathrm{RCF}(L_7^\lambda) = \mathrm{RCF}(K_7^\lambda)$,
and conclude that there are no new identities for $\lambda$.
These equalities hold for all partitions $\lambda$; see Table \ref{PL7table}.
\end{proof}

\begin{conjecture}
Over a field of characteristic 0, every polynomial identity of degree $n > 5$ for the pre-Lie triple products 
in the free dendriform dialgebra
follows from the ternary right-symmetric identity and the three identities of Theorem \ref{PL5}.
\end{conjecture}


\section{Pre-Jordan triple systems} \label{prejordansection}

We will now write the pre-Jordan product in a dendriform dialgebra as $a \cdot b = a \succ b + b \prec a$,
and the pre-Jordan triple products as
  \[
  \langle a, b, c \rangle_1 = (a \circ b) \circ c,
  \qquad
  \langle a, b, c \rangle_2 = a \circ (b \circ c).
  \]
The polynomial identities of degree $n \le 4$ satisfied by the pre-Jordan product in the free dendriform algebra
define the variety of pre-Jordan algebras, introduced by Hou, Ni and Bai \cite{HouNiBai2013}.
Special identities of degree 8 for the pre-Jordan product have been studied in \cite{BremnerMadariaga2013}.

Consider a Jordan algebra $J$ with operation $a \circ b$; although the product is commutative,
we use both left and right multiplications, $L_a(b) = a \circ b$ and $R_c(b) = b \circ c$.
The Jordan triple product is defined by
  \[
  (a \circ b) \circ c + a \circ (b \circ c) - (a \circ c) \circ b 
  =
  ( L_a R_c + R_c L_a )(b) - L_{a \circ c}(b).
  \]
This measures the difference in $\mathrm{End}_\mathbb{F}(J)$ between the anticommutator of $L_a$ and $R_c$,
and (left) multiplication by $a \circ c$.
If we replace $a \circ b$ by the noncommutative pre-Jordan product then we must consider two distinct operations:
  \begin{align}
  \label{pjtp1}
  ( a \cdot b ) \cdot c + a \cdot ( b \cdot c ) - ( a \cdot c ) \cdot b
  =
  ( L_a R_c + R_c L_a )(b) - L_{a \circ c}(b),
  \\  
  \label{pjtp2}
  ( a \cdot b ) \cdot c + a \cdot ( b \cdot c ) - b \cdot ( c \cdot a )
  =
  ( L_a R_c + R_c L_a )(b) - R_{c \circ a}(b).
  \end{align}

\begin{lemma} \label{newpjtp}
Let $J \subset \DD_3$ be the $S_3$-submodule generated by operations \eqref{pjtp1}-\eqref{pjtp2}.
Then $\dim J = 12$ and hence the operations satisfy no identities in degree 3.
\end{lemma}

\begin{proof}
Similar to the proof of Lemma \ref{PL3}, except that we replace the pre-Lie triple products
by operations \eqref{pjtp1}-\eqref{pjtp2}, and find that the expansion matrix $E_3$ has full rank.
If $J_i$ ($i=1,2$) is the submodule generated by operation $i$, then clearly $\dim J_i \le |S_3| = 6$,
and since $\dim (J_1 + J_2) = 12$ we have $J_1 \cap J_2 = \{0\}$.
\end{proof}

Clearly the submodule $J$ of Lemma \ref{newpjtp} is contained in the submodule generated by $\langle a, b, c \rangle_i$ 
($i=1,2$), but the latter has dimension $\le 12$, so the two are equal.
For simplicity, we use $\langle a, b, c \rangle_i$ ($i=1,2$) instead of the operations \eqref{pjtp1}-\eqref{pjtp2},
which would be more natural following a strict analogy with the classical case.
Then Lemma \ref{newpjtp} implies that the pre-Jordan triple products satisfy no identities in degree 3.

\begin{theorem} \label{theoremPJ5}
Let $\mathcal{E}_5\colon \TT_5 \to \DD_5$ be the expansion map for the pre-Jordan triple products,
and let $E_5$ be the expansion matrix representing $\mathcal{E}_5$
with respect to the ordered bases of monomials.
The nullspace $N_5$ of $E_5$ has dimension 335, and $S_5$-module generators
are the identities $\mathrm{PJTS_k}(a,b,c,d,e) \equiv 0$ $(k=1,\dots,5)$, 
which are independent in the sense that none is a consequence of the others:
  \begin{align*}
  \mathrm{PJTS}_1
  &=
    \langle \langle a, b, c \rangle_1, d, e \rangle_1  
  + \langle \langle b, a, c \rangle_1, d, e \rangle_1  
  + \langle \langle c, a, b \rangle_2, d, e \rangle_1  
  + \langle \langle c, b, a \rangle_2, d, e \rangle_1  
  \\
  &\quad
  - \langle a, \langle b, c, d \rangle_1, e \rangle_1  
  - \langle a, \langle c, b, d \rangle_1, e \rangle_1  
  - \langle b, \langle a, c, d \rangle_1, e \rangle_1  
  - \langle b, \langle c, a, d \rangle_1, e \rangle_1  
  \\
  &\quad
  - \langle c, \langle a, b, d \rangle_1, e \rangle_1  
  - \langle c, \langle b, a, d \rangle_1, e \rangle_1  
  + \langle a, \langle c, b, d \rangle_2, e \rangle_1  
  + \langle b, \langle c, a, d \rangle_2, e \rangle_1, 
  \\ 
  \mathrm{PJTS}_2
  &=
    \langle \langle b, d, a \rangle_1, c, e \rangle_2  
  + \langle \langle b, d, c \rangle_1, a, e \rangle_2  
  + \langle \langle a, b, d \rangle_2, c, e \rangle_2  
  + \langle \langle c, b, d \rangle_2, a, e \rangle_2  
  \\
  &\quad
  - \langle a, \langle b, d, c \rangle_1, e \rangle_2  
  - \langle c, \langle b, d, a \rangle_1, e \rangle_2  
  - \langle a, \langle c, b, d \rangle_2, e \rangle_2  
  - \langle c, \langle a, b, d \rangle_2, e \rangle_2  
  \\
  &\quad
  + \langle a, c, \langle b, d, e \rangle_1 \rangle_1  
  - \langle b, d, \langle a, c, e \rangle_1 \rangle_1  
  - \langle b, d, \langle c, a, e \rangle_1 \rangle_1  
  + \langle c, a, \langle b, d, e \rangle_1 \rangle_1, 
  \\ 
  \mathrm{PJTS}_3
  &=
    \langle \langle b, c, d \rangle_1, a, e \rangle_1  
  + \langle \langle d, b, c \rangle_2, a, e \rangle_1  
  + \langle a, \langle b, c, d \rangle_1, e \rangle_1  
  - \langle a, \langle b, c, d \rangle_1, e \rangle_2  
  \\
  &\quad
  - \langle d, \langle b, c, a \rangle_1, e \rangle_2  
  + \langle a, \langle d, b, c \rangle_2, e \rangle_1  
  - \langle a, \langle d, b, c \rangle_2, e \rangle_2  
  - \langle d, \langle a, b, c \rangle_2, e \rangle_2  
  \\
  &\quad
  - \langle b, c, \langle a, d, e \rangle_1 \rangle_1  
  - \langle b, c, \langle d, a, e \rangle_1 \rangle_1  
  + \langle d, a, \langle b, c, e \rangle_1 \rangle_2  
  + \langle b, c, \langle a, d, e \rangle_2 \rangle_1, 
  \\ 
  \mathrm{PJTS}_4
  &=
    \langle a, \langle b, c, d \rangle_1, e \rangle_2  
  + \langle a, \langle c, b, d \rangle_1, e \rangle_2  
  + \langle a, \langle d, b, c \rangle_2, e \rangle_2  
  + \langle a, \langle d, c, b \rangle_2, e \rangle_2  
  \\
  &\quad
  - \langle a, b, \langle c, d, e \rangle_1 \rangle_2  
  - \langle a, b, \langle d, c, e \rangle_1 \rangle_2  
  - \langle a, c, \langle b, d, e \rangle_1 \rangle_2  
  - \langle a, c, \langle d, b, e \rangle_1 \rangle_2  
  \\
  &\quad
  - \langle a, d, \langle b, c, e \rangle_1 \rangle_2  
  - \langle a, d, \langle c, b, e \rangle_1 \rangle_2  
  + \langle a, b, \langle d, c, e \rangle_2 \rangle_2  
  + \langle a, c, \langle d, b, e \rangle_2 \rangle_2, 
  \\ 
  \mathrm{PJTS}_5
  &=
    \langle \langle a, c, b \rangle_1, d, e \rangle_2  
  + \langle \langle a, c, d \rangle_1, b, e \rangle_2  
  + \langle \langle c, a, b \rangle_1, d, e \rangle_2  
  + \langle \langle c, a, d \rangle_1, b, e \rangle_2  
  \\
  &\quad
  + \langle \langle c, d, b \rangle_1, a, e \rangle_2  
  + \langle \langle d, a, b \rangle_1, c, e \rangle_2  
  + \langle \langle d, a, c \rangle_1, b, e \rangle_2  
  + \langle \langle d, b, a \rangle_1, c, e \rangle_2  
  \\
  &\quad
  + \langle \langle d, b, c \rangle_1, a, e \rangle_2  
  + \langle \langle d, c, b \rangle_1, a, e \rangle_2  
  + \langle \langle a, d, b \rangle_2, c, e \rangle_2  
  + \langle \langle b, a, c \rangle_2, d, e \rangle_2  
  \\
  &\quad
  + \langle \langle b, c, a \rangle_2, d, e \rangle_2  
  + \langle \langle b, c, d \rangle_2, a, e \rangle_2  
  + \langle \langle b, d, a \rangle_2, c, e \rangle_2  
  + \langle \langle b, d, c \rangle_2, a, e \rangle_2  
  \\
  &\quad
  + \langle \langle c, d, a \rangle_2, b, e \rangle_2  
  + \langle \langle c, d, b \rangle_2, a, e \rangle_2  
  + \langle \langle d, a, c \rangle_2, b, e \rangle_2  
  + \langle \langle d, c, a \rangle_2, b, e \rangle_2  
  \\
  &\quad
  - \langle a, \langle b, c, d \rangle_1, e \rangle_1  
  - \langle a, \langle b, d, c \rangle_1, e \rangle_1  
  - \langle a, \langle c, b, d \rangle_1, e \rangle_1  
  - \langle a, \langle c, d, b \rangle_1, e \rangle_1  
  \\
  &\quad
  - \langle a, \langle d, b, c \rangle_1, e \rangle_1  
  - \langle a, \langle d, c, b \rangle_1, e \rangle_1  
  - \langle b, \langle a, c, d \rangle_1, e \rangle_1  
  - \langle b, \langle a, d, c \rangle_1, e \rangle_1  
  \\
  &\quad
  - \langle b, \langle c, a, d \rangle_1, e \rangle_1  
  - \langle b, \langle c, d, a \rangle_1, e \rangle_1  
  - \langle b, \langle d, a, c \rangle_1, e \rangle_1  
  - \langle b, \langle d, c, a \rangle_1, e \rangle_1  
  \\
  &\quad
  - \langle c, \langle a, b, d \rangle_1, e \rangle_1  
  - \langle c, \langle a, d, b \rangle_1, e \rangle_1  
  - \langle c, \langle b, a, d \rangle_1, e \rangle_1  
  - \langle c, \langle b, d, a \rangle_1, e \rangle_1  
  \\
  &\quad
  - \langle c, \langle d, a, b \rangle_1, e \rangle_1  
  - \langle c, \langle d, b, a \rangle_1, e \rangle_1  
  - \langle d, \langle a, b, c \rangle_1, e \rangle_1  
  - \langle d, \langle a, c, b \rangle_1, e \rangle_1  
  \\
  &\quad
  - \langle d, \langle b, a, c \rangle_1, e \rangle_1  
  - \langle d, \langle b, c, a \rangle_1, e \rangle_1  
  - \langle d, \langle c, a, b \rangle_1, e \rangle_1  
  - \langle d, \langle c, b, a \rangle_1, e \rangle_1  
  \\
  &\quad
  - \langle a, \langle d, b, c \rangle_1, e \rangle_2  
  + \langle b, \langle a, d, c \rangle_1, e \rangle_2  
  + \langle c, \langle a, d, b \rangle_1, e \rangle_2  
  - \langle c, \langle d, b, a \rangle_1, e \rangle_2  
  \\
  &\quad
  - \langle a, \langle c, d, b \rangle_2, e \rangle_2  
  + \langle b, \langle c, a, d \rangle_2, e \rangle_2  
  - \langle c, \langle a, d, b \rangle_2, e \rangle_2  
  + \langle c, \langle b, a, d \rangle_2, e \rangle_2  
  \\
  &\quad
  + \langle a, b, \langle c, d, e \rangle_1 \rangle_1  
  + \langle a, b, \langle d, c, e \rangle_1 \rangle_1  
  + \langle a, c, \langle d, b, e \rangle_1 \rangle_1  
  + \langle a, d, \langle b, c, e \rangle_1 \rangle_1  
  \\
  &\quad
  + \langle a, d, \langle c, b, e \rangle_1 \rangle_1  
  + \langle b, a, \langle c, d, e \rangle_1 \rangle_1  
  + \langle b, a, \langle d, c, e \rangle_1 \rangle_1  
  + \langle b, c, \langle d, a, e \rangle_1 \rangle_1  
  \\
  &\quad
  + \langle b, d, \langle a, c, e \rangle_1 \rangle_1  
  + \langle b, d, \langle c, a, e \rangle_1 \rangle_1  
  + \langle c, a, \langle d, b, e \rangle_1 \rangle_1  
  + \langle c, b, \langle d, a, e \rangle_1 \rangle_1  
  \\
  &\quad
  - \langle b, a, \langle c, d, e \rangle_1 \rangle_2  
  - \langle b, a, \langle d, c, e \rangle_1 \rangle_2  
  - \langle b, c, \langle a, d, e \rangle_1 \rangle_2  
  - \langle b, c, \langle d, a, e \rangle_1 \rangle_2  
  \\
  &\quad
  - \langle b, d, \langle a, c, e \rangle_1 \rangle_2  
  - \langle b, d, \langle c, a, e \rangle_1 \rangle_2  
  - \langle c, a, \langle d, b, e \rangle_2 \rangle_2  
  - \langle d, a, \langle c, b, e \rangle_2 \rangle_2. 
  \end{align*}
\end{theorem}

\begin{proof}
In degree 5, there are 42 normal DD-types and 12 TT-types, so the
expansion matrix $E_5$ has size $5040 \times 1440$;
the $(i,j)$-entry is the coefficient of the $i$-th DD-monomial in the normalized expansion 
of the $j$-th TT-monomial.
Using modular arithmetic ($p = 101$), we compute the row canonical form of $E_5$, 
and find that its rank is 1105, so its nullspace $N_5$ has dimension 335.
We extract the canonical basis of $N_5$, find (using symmetric representatives modulo $p$)
that all the coefficients are in $\{0,1,-1\}$, interpret these coefficients as integers, 
and sort the 335 identities by increasing Euclidean length.
Further calculations with the module generators algorithm (Bremner and Peresi \cite[Fig.~3]{BremnerPeresi2009})
show that $N_5$ is generated as an $S_5$-module by the stated identities, 
and that these identities are independent.
\end{proof}

\begin{table}
\begin{tabular}{lr|rrr|rrrrr}
&\; & \multicolumn{3}{c|}{\; $L_7^\lambda$ (lifted identities)} & \multicolumn{5}{c}{\; $X_7^\lambda$ (all identities)} 
\\
$\lambda$ &\; $d_\lambda$ &\; rows &\; cols &\; rank &\; rows &\; cols &\; rank &\; null &\; new 
\\ 
\midrule
  7        &   1 &    80 &    96 &    61 &    96 &    429 &    368 &    61 &  0 \\ 
  61       &   6 &   480 &   576 &   322 &   576 &   2574 &   2252 &   322 &  0 \\ 
  52       &  14 &  1120 &  1344 &   680 &  1344 &   6006 &   5326 &   680 &  0 \\ 
  511      &  15 &  1200 &  1440 &   711 &  1440 &   6435 &   5724 &   711 &  0 \\ 
  43       &  14 &  1120 &  1344 &   638 &  1344 &   6006 &   5368 &   638 &  0 \\ 
  421      &  35 &  2800 &  3360 &  1527 &  3360 &  15015 &  13487 &  1528 &  $\rightarrow$ 1 \\ 
  4111     &  20 &  1600 &  1920 &   836 &  1920 &   8580 &   7744 &   836 &  0 \\ 
  331      &  21 &  1680 &  2016 &   877 &  2016 &   9009 &   8131 &   878 &  $\rightarrow$ 1 \\ 
  322      &  21 &  1680 &  2016 &   846 &  2016 &   9009 &   8162 &   847 &  $\rightarrow$ 1 \\ 
  3211     &  35 &  2800 &  3360 &  1368 &  3360 &  15015 &  13645 &  1370 &  $\rightarrow$ 2 \\ 
  31111    &  15 &  1200 &  1440 &   546 &  1440 &   6435 &   5888 &   547 &  $\rightarrow$ 1 \\ 
  2221     &  14 &  1120 &  1344 &   512 &  1344 &   6006 &   5493 &   513 &  $\rightarrow$ 1 \\ 
  22111    &  14 &  1120 &  1344 &   484 &  1344 &   6006 &   5520 &   486 &  $\rightarrow$ 2 \\ 
  211111   &   6 &   480 &   576 &   184 &   576 &   2574 &   2389 &   185 &  $\rightarrow$ 1 \\ 
  1111111  &   1 &    80 &    96 &    24 &    96 &    429 &    405 &    24 &  0 \\ 
\midrule
\end{tabular}
\medskip
\caption{Ranks for pre-Jordan triple products in degree 7}
\label{PJ7table}
\end{table}

\begin{theorem}
In degree 7, there exist polynomial identities for the pre-Jordan triple products which do not follow from the identities
of lower degree.
\end{theorem}

\begin{proof}
The method is the same as in the proof of Theorem \ref{theoremPL7}, but in this case we obtain new identities in degree 7.
These new identities occur for the partitions $\lambda$ for which the nullity of the expansion matrix $X_7^\lambda$ 
is strictly greater than the rank of the matrix $L_7^\lambda$ of lifted identities.
Eight representations have either one or two new identities, for a total of ten irreducible identities.
The details are displayed in Table \ref{PJ7table}; see especially the rightmost column, labelled ``new''.
\end{proof}

\begin{table}
  \begin{align*}
  &
  \big[ \, 
     2  D_{ 1, 15 }  
  \, \big]_{ 25 } 
  +
  \big[ \, 
     2  D_{ 1, 14 }  
  {-}   D_{ 1, 15 }  
  \, \big]_{ 26 } 
  +
  \big[ \, 
  {-}   D_{ 1, 12 }  
  {-}2  D_{ 1, 15 }  
  \, \big]_{ 28 } 
  +
  \big[ \, 
     2  D_{ 1, 14 }  
  {+}   D_{ 1, 15 }  
  \, \big]_{ 30 } 
  \\
  &
  +
  \big[ \, 
        D_{ 1, 12 }  
  {-}2  D_{ 1, 15 }  
  \, \big]_{ 32 } 
  +
  \big[ \, 
  {-}4  D_{ 1,  9 }  
  {-}4  D_{ 1, 14 }  
  \, \big]_{ 35 } 
  +
  \big[ \, 
     2  D_{ 1,  9 }  
  {-}   D_{ 1, 15 }  
  \, \big]_{ 36 } 
  +
  \big[ \, 
  {-}2  D_{ 1, 15 }  
  \, \big]_{ 37 } 
  \\
  &
  +
  \big[ \, 
  {-}2  D_{ 1,  5 }  
  {+}2  D_{ 1,  9 }  
  {+}2  D_{ 1, 12 }  
  {+}2  D_{ 1, 14 }  
  \, \big]_{ 39 } 
  +
  \big[ \, 
     2  D_{ 1,  5 }  
  {-}2  D_{ 1,  9 }  
  {+}3  D_{ 1, 15 }  
  \, \big]_{ 40 } 
  +
  \big[ \, 
     4  D_{ 1,  5 }  
  \, \big]_{ 43 } 
  \\
  &
  +
  \big[ \, 
  {-}   D_{ 1, 15 }  
  \, \big]_{ 44 } 
  +
  \big[ \, 
     2  D_{ 1, 15 }  
  \, \big]_{ 45 } 
  +
  \big[ \, 
  {-}   D_{ 1, 15 }  
  \, \big]_{ 48 } 
  +
  \big[ \, 
     4  D_{ 1, 12 }  
  \, \big]_{ 49 } 
  \\
  &
  +
  \big[ \, 
     2  D_{ 1,  5 }  
  {+}2  D_{ 1,  9 }  
  {+}2  D_{ 1, 14 }  
  {-}   D_{ 1, 15 }  
  \, \big]_{ 50 } 
  +
  \big[ \, 
  {-}2  D_{ 1,  5 }  
  {-}2  D_{ 1, 15 }  
  \, \big]_{ 51 } 
  +
  \big[ \, 
  {-}2  D_{ 1,  9 }  
  \, \big]_{ 52 } 
  \\
  &
  +
  \big[ \, 
  {-}   D_{ 1,  5 }  
  {-}   D_{ 1,  9 }  
  {-}2  D_{ 1, 14 }  
  {+}2  D_{ 1, 15 }  
  \, \big]_{ 53 } 
  +
  \big[ \, 
        D_{ 1,  5 }  
  {-}3  D_{ 1,  9 }  
  {-}2  D_{ 1, 14 }  
  \, \big]_{ 54 } 
  \\
  &
  +
  \big[ \, 
     3  D_{ 1,  5 }  
  {-}3  D_{ 1, 12 }  
  \, \big]_{ 55 } 
  +
  \big[ \, 
  {-}   D_{ 1,  9 }  
  \, \big]_{ 56 } 
  +
  \big[ \, 
     2  D_{ 1, 12 }  
  {+}2  D_{ 1, 15 }  
  \, \big]_{ 58 } 
  +
  \big[ \, 
     2  D_{ 1, 12 }  
  {-}2  D_{ 1, 14 }  
  \, \big]_{ 62 } 
  \\
  &
  +
  \big[ \, 
     2  D_{ 1, 12 }  
  {+}2  D_{ 1, 14 }  
  {-}2  D_{ 1, 15 }  
  \, \big]_{ 63 } 
  +
  \big[ \, 
  {-}4  D_{ 1, 12 }  
  {+}4  D_{ 1, 15 }  
  \, \big]_{ 64 } 
  +
  \big[ \, 
     2  D_{ 1,  9 }  
  {-}   D_{ 1, 12 }  
  {+}2  D_{ 1, 14 }  
  \, \big]_{ 66 } 
  \\
  &
  +
  \big[ \, 
     4  D_{ 1,  9 }  
  {-}4  D_{ 1, 12 }  
  {+}2  D_{ 1, 14 }  
  {-}   D_{ 1, 15 }  
  \, \big]_{ 68 } 
  +
  \big[ \, 
  {-}3  D_{ 1,  5 }  
  {+}2  D_{ 1,  9 }  
  {-}4  D_{ 1, 12 }  
  {-}6  D_{ 1, 15 }  
  \, \big]_{ 70 } 
  \\
  &
  +
  \big[ \, 
     2  D_{ 1,  5 }  
  {-}2  D_{ 1,  9 }  
  {+}2  D_{ 1, 12 }  
  {-}4  D_{ 1, 14 }  
  {+}4  D_{ 1, 15 }  
  \, \big]_{ 71 } 
  +
  \big[ \, 
  {-}4  D_{ 1,  5 }  
  {-}4  D_{ 1, 12 }  
  {+}2  D_{ 1, 14 }  
  {-}   D_{ 1, 15 }  
  \, \big]_{ 72 } 
  \\
  &
  +
  \big[ \, 
  {-}   D_{ 1, 15 }  
  \, \big]_{ 75 } 
  +
  \big[ \, 
  {-}2  D_{ 1,  5 }  
  {+}2  D_{ 1, 15 }  
  \, \big]_{ 76 } 
  +
  \big[ \, 
     8  D_{ 1,  5 }  
  {-}4  D_{ 1,  9 }  
  {+}4  D_{ 1, 12 }  
  {-}4  D_{ 1, 14 }  
  {+}6  D_{ 1, 15 }  
  \, \big]_{ 77 } 
  \\
  &
  +
  \big[ \, 
  {-}6  D_{ 1,  5 }  
  {+}2  D_{ 1,  9 }  
  {-}2  D_{ 1, 12 }  
  {+}2  D_{ 1, 14 }  
  {+}2  D_{ 1, 15 }  
  \, \big]_{ 78 } 
  \\
  &
  +
  \big[ \, 
  {-}4  D_{ 1,  5 }  
  {+}2  D_{ 1,  9 }  
  {-}2  D_{ 1, 12 }  
  {+}2  D_{ 1, 14 }  
  {-}5  D_{ 1, 15 }  
  \, \big]_{ 79 } 
  +
  \big[ \, 
     4  D_{ 1,  5 }  
  {-}4  D_{ 1, 15 }  
  \, \big]_{ 80 } 
  \\
  &
  +
  \big[ \, 
  {-}2  D_{ 1,  5 }  
  {+}   D_{ 1,  9 }  
  {-}   D_{ 1, 15 }  
  \, \big]_{ 83 } 
  +
  \big[ \, 
  {-}   D_{ 1,  9 }  
  {-}2  D_{ 1, 14 }  
  {+}2  D_{ 1, 15 }  
  \, \big]_{ 84 } 
  \\
  &
  +
  \big[ \, 
  {-}8  D_{ 1,  5 }  
  {+}6  D_{ 1,  9 }  
  {-}4  D_{ 1, 12 }  
  {+}4  D_{ 1, 14 }  
  {-}4  D_{ 1, 15 }  
  \, \big]_{ 85 } 
  \\
  &
  +
  \big[ \, 
     6  D_{ 1,  5 }  
  {+}2  D_{ 1,  9 }  
  {+}2  D_{ 1, 12 }  
  {-}2  D_{ 1, 14 }  
  {+}2  D_{ 1, 15 }  
  \, \big]_{ 86 } 
  \\
  &
  +
  \big[ \, 
     2  D_{ 1,  5 }  
  {-}   D_{ 1,  9 }  
  {+}   D_{ 1, 12 }  
  {-}2  D_{ 1, 14 }  
  {+}2  D_{ 1, 15 }  
  \, \big]_{ 87 } 
  +
  \big[ \, 
  {-}   D_{ 1,  9 }  
  {+}2  D_{ 1, 12 }  
  {+}2  D_{ 1, 15 }  
  \, \big]_{ 88 } 
  \\
  &
  +
  \big[ \, 
     2  D_{ 1,  5 }  
  {-}2  D_{ 1,  9 }  
  {-}2  D_{ 1, 12 }  
  {-}   D_{ 1, 14 }  
  {+}   D_{ 1, 15 }  
  \, \big]_{ 89 } 
  +
  \big[ \, 
        D_{ 1, 14 }  
  {+}   D_{ 1, 15 }  
  \, \big]_{ 90 } 
  \\
  &
  +
  \big[ \, 
     2  D_{ 1,  5 }  
  {+}   D_{ 1, 12 }  
  {-}   D_{ 1, 14 }  
  {-}   D_{ 1, 15 }  
  \, \big]_{ 91 } 
  +
  \big[ \, 
  {-}3  D_{ 1,  5 }  
  {+}   D_{ 1, 12 }  
  {+}   D_{ 1, 14 }  
  {-}   D_{ 1, 15 }  
  \, \big]_{ 92 } 
  \\
  &
  +
  \big[ \, 
  {-}2  D_{ 1,  5 }  
  {-}   D_{ 1,  9 }  
  {-}   D_{ 1, 12 }  
  {+}   D_{ 1, 14 }  
  \, \big]_{ 93 } 
  +
  \big[ \, 
     3  D_{ 1,  5 }  
  {-}2  D_{ 1,  9 }  
  {+}2  D_{ 1, 12 }  
  {-}   D_{ 1, 14 }  
  \, \big]_{ 94 } 
  \\
  &
  +
  \big[ \, 
        D_{ 1, 12 }  
  {+}   D_{ 1, 15 }  
  \, \big]_{ 95 } 
  +
  \big[ \, 
  {-}2  D_{ 1,  5 }  
  {+}3  D_{ 1, 12 }  
  {-}   D_{ 1, 15 }  
  \, \big]_{96} 
  \end{align*} 
\caption{New irreducible pre-Jordan triple product identity for partition $\lambda = 31111$}
\label{newPJ7identity}
\end{table}

\begin{corollary} \label{PJ7partition31111}
The polynomial identity in Table \ref{newPJ7identity} is an irreducible identity for partition 31111 which does
not follow from the identities of Theorem \ref{theoremPJ5}.
\end{corollary}

\begin{proof}
We follow the method of Bremner and Peresi \cite[\S 4]{BremnerPeresi2012}.
For partition $\lambda = 31111$, we know from Table \ref{PJ7table} that $\mathrm{RCF}(L_7^\lambda)$ has 546 nonzero rows, 
and that $\mathrm{RCF}(K_7^\lambda)$ has 547 nonzero rows.
Furthermore, the row space of $\mathrm{RCF}(L_7^\lambda)$ is a subspace of the row space of $\mathrm{RCF}(K_7^\lambda)$.
Let $\mathrm{leading}(L_7^\lambda)$ and $\mathrm{leading}(K_7^\lambda)$ denote respectively the sets of indices
of the columns of $\mathrm{RCF}(L_7^\lambda)$ and $\mathrm{RCF}(K_7^\lambda)$ which contain leading 1s of the
nonzero rows.
We have $\mathrm{leading}(L_7^\lambda) \subset \mathrm{leading}(K_7^\lambda)$, and we find that
  \[
  \mathrm{leading}(K_7^\lambda) - \mathrm{leading}(L_7^\lambda) = \{ 375 \};
  \]
row 225 of $\mathrm{RCF}(K_7^\lambda)$ has its leading 1 in column 375.
Each row of $\mathrm{RCF}(K_7^\lambda)$ is a vector with 1440 entries, divided into 96 groups of 15 entries,
where the groups correspond to the 96 TT-types, and the entries in each group correspond to 
the $d_\lambda = 15$ standard Young tableaux $T_1, \dots, T_{15}$ for partition $\lambda = 31111$:
  \[
  \begin{array}{l} 
  123 \\[-2pt] 
  4   \\[-2pt] 
  5   \\[-2pt] 
  6   \\[-2pt] 
  7
  \end{array} 
  \!
  \begin{array}{l} 
  124 \\[-2pt] 
  3   \\[-2pt] 
  5   \\[-2pt] 
  6   \\[-2pt] 
  7
  \end{array} 
  \!
  \begin{array}{l} 
  125 \\[-2pt] 
  3   \\[-2pt] 
  4   \\[-2pt] 
  6   \\[-2pt] 
  7
  \end{array} 
  \!
  \begin{array}{l} 
  126 \\[-2pt] 
  3   \\[-2pt] 
  4   \\[-2pt] 
  5   \\[-2pt] 
  7
  \end{array} 
  \!
  \begin{array}{l} 
  127 \\[-2pt] 
  3   \\[-2pt] 
  4   \\[-2pt] 
  5   \\[-2pt] 
  6
  \end{array} 
  \!
  \begin{array}{l} 
  134 \\[-2pt] 
  2   \\[-2pt] 
  5   \\[-2pt] 
  6   \\[-2pt] 
  7
  \end{array} 
  \!
  \begin{array}{l} 
  135 \\[-2pt] 
  2   \\[-2pt] 
  4   \\[-2pt] 
  6   \\[-2pt] 
  7
  \end{array} 
  \!
  \begin{array}{l} 
  136 \\[-2pt] 
  2   \\[-2pt] 
  4   \\[-2pt] 
  5   \\[-2pt] 
  7
  \end{array} 
  \!
  \begin{array}{l} 
  137 \\[-2pt] 
  2   \\[-2pt] 
  4   \\[-2pt] 
  5   \\[-2pt] 
  6
  \end{array} 
  \!
  \begin{array}{l} 
  145 \\[-2pt] 
  2   \\[-2pt] 
  3   \\[-2pt] 
  6   \\[-2pt] 
  7
  \end{array} 
  \!
  \begin{array}{l} 
  146 \\[-2pt] 
  2   \\[-2pt] 
  3   \\[-2pt] 
  5   \\[-2pt] 
  7
  \end{array} 
  \!
  \begin{array}{l} 
  147 \\[-2pt] 
  2   \\[-2pt] 
  3   \\[-2pt] 
  5   \\[-2pt] 
  6
  \end{array} 
  \!
  \begin{array}{l} 
  156 \\[-2pt] 
  2   \\[-2pt] 
  3   \\[-2pt] 
  4   \\[-2pt] 
  7
  \end{array} 
  \!
  \begin{array}{l} 
  157 \\[-2pt] 
  2   \\[-2pt] 
  3   \\[-2pt] 
  4   \\[-2pt] 
  6
  \end{array} 
  \!
  \begin{array}{l} 
  167 \\[-2pt] 
  2   \\[-2pt] 
  3   \\[-2pt] 
  4   \\[-2pt] 
  5
  \end{array} 
  \]
For $i, j = 1, \dots, d_\lambda$ let $s_{ij} \in S_7$ be such that $s_{ij} T_i = T_j$.
For $j = 1, \dots, d_\lambda$ let $R_j$ and $C_j$ respectively be the subgroups of $S_7$ which
leave the rows and columns of $T_j$ fixed as sets.
We define elements of the group algebra $\mathbb{F} S_7$ as follows:
  \[
  D_{ii} = \frac{d_\lambda}{n!} \sum_{\sigma \in R_i} \sum_{\tau \in C_i} \epsilon(\tau) \sigma \tau,
  \qquad
  D_{ij} = D_{ii} s_{ij}^{-1}.
  \]
Each of these elements is a linear combination of $3!\cdot 5! = 720$ permutations.
For this partition, $\lambda = 31111$, the matrix computed by the algorithm of Clifton \cite{Clifton1981}
for the identity permutation is in fact the identity matrix, and so the elements $D_{ij}$ satisfy
the matrix unit relations, $D_{ij} D_{k\ell} = \delta_{jk} E_{i\ell}$.
In Table \ref{newPJ7identity}, the notation $[ X ]_k$ indicates the application of TT-type $k$
to the element $X \in \mathbb{F} S_7$; that table gives the polynomial identity represented by 
row 225 of $\mathrm{RCF}(K_7^\lambda)$.
This calculation was done using modular arithmetic ($p = 101$), 
and the entries of row 225 of $\mathrm{RCF}(K_7^\lambda)$ were
$\{ 1, 2, 3, 4, 48, 49, 50, 51, 52, 97, 98, 99, 100 \}$.
We multiplied each entry by 2 and used symmetric representatives to obtain the coefficients in Table \ref{newPJ7identity}.
\end{proof}


\section{Conclusion}

The KP and BSO algorithms \cite{BFSO2012,BSO2010,Kolesnikov2008,Pozhidaev2010} show how to extend polynomial identities
and multilinear operations for a variety of algebras to the corresponding identities and operations in the setting of
binary and $n$-ary dialgebras.
The work of Gubarev and Kolesnikov \cite{GubarevKolesnikov2013} extends the KP algorithm in the binary case to the setting 
of dendriform dialgebras.
One open problem is to generalize their work to the $n$-ary case, with the goal of finding other natural analogues of 
Lie and Jordan triple system in the dendriform setting; another is to extend the BSO algorithm to the dendriform setting. 
The natural framework for this direction of research is the theory of Koszul duality for operads introduced by Ginzburg
and Kapranov \cite{GinzburgKapranov}.


\section*{Acknowledgements}

Murray Bremner was partially supported by a Discovery Grant from NSERC
(Natural Sciences and Engineering Research Council); he thanks the 
Department of Mathematics and Computer Science at the University of La Rioja in Logro\~no
for its hospitality in April and May 2013.
Sara Madariaga was partially supported by a Postdoctoral Fellowship from PIMS
(Pacific Institute for Mathematical Sciences).


\end{document}